\documentclass[a4paper,11pt]{article}
\usepackage[cp1251]{inputenc}
\usepackage[french, english]{babel}
\usepackage{amsfonts}
\usepackage{amsthm}
\usepackage{amsbsy}
\usepackage{amssymb}
\usepackage{graphicx}
\usepackage{eso-pic}
\usepackage{dsfont}
\usepackage{tikz}
\usepackage{xfrac}
\usepackage{float}
\usepackage[bookmarks=true]{hyperref}
\usepackage{bookmark}
\usepackage{amsmath}
\usepackage{subfigure}
\usepackage{hyperref}
\usetikzlibrary{matrix,arrows}

\usepackage[letterpaper,left=0.8in,right=0.8in,top=1in,bottom=1in]{geometry}
\newtheorem{lemma}{Lemma}[section]
\newtheorem{teo}[lemma]{Theorem}
\newtheorem{prop}[lemma]{Proposition}

\theoremstyle{definition}
\newtheorem{defn}[lemma]{Definition}

\newtheorem{example}[lemma]{Example}

\theoremstyle{remark}
\newtheorem{rem}[lemma]{Remark}

\newcommand{\calX} {\ensuremath {\mathcal{X}}}

\newcommand{\calN} {\ensuremath {\mathcal{N}}}

\newcommand{\calM} {\ensuremath {\mathcal{M}}}
\newcommand{\calC} {\ensuremath {\mathcal{C}}}
\newcommand{\calA} {\ensuremath {\mathcal{A}}}

\newcommand{\calD} {\ensuremath {\mathcal{D}}}
\newcommand{\calF} {\ensuremath {\mathcal{F}}}
\newcommand{\calP} {\ensuremath {\mathcal{P}}}

\newcommand{\calY}{\ensuremath {\mathcal{Y}}}
\newcommand{\calH}{\ensuremath {\mathcal{H}}}

\newcommand{\calG}{\ensuremath {\mathcal{G}}}

\begin{document}
\title{Hyperbolic $4$-manifolds, colourings and mutations}
\author{Alexander Kolpakov \& Leone Slavich}
\date{}
\maketitle

\selectlanguage{french} 
\begin{abstract}
\noindent Nous proposons une nouvelle approche \`a construire des vari\'et\'es hyperboliques $\calM$ en dimension quatre, par moyens d'un poly\`edre de Coxeter $\calP \subset \mathbb{H}^4$ munit avec un coloriage de ses faces. Aussi, nous utilisons notre m\'ethode pour obtenir des sous-surfaces totalement g\'eod\'esiques plong\'ees dans $\calM$, et pour d\'ecrire le r\'esultat de mutations par rapport \`a ses surfaces. Comme application, nous construisons une vari\'et\'e compl\`ete hyperbolique non-compacte $\calX$ avec un bout cuspide non-torique, et une vari\'et\'e compl\`ete hyperbolique non-compacte $\calY$ avec un bout cuspide torique. Elles sont des nouveaux exemples de vari\'et\'es hyperboliques compl\`etes non-compactes en dimension quatre avec un seul bout cuspide et volume relativement petit.  
\end{abstract}

\selectlanguage{english}

\begin{abstract}
\noindent We develop a way of seeing a complete orientable hyperbolic $4$-manifold $\calM$ as an orbifold cover of a Coxeter polytope $\calP \subset \mathbb{H}^4$ that has a facet colouring. We also develop a way of finding a totally geodesic sub-manifold $\calN$ in $\calM$, and describing the result of mutations along $\calN$. As an application of our method, we construct an example of a complete orientable hyperbolic $4$-manifold $\calX$ with a single non-toric cusp and a complete orientable hyperbolic $4$-manifold $\calY$ with a single toric cusp. Both $\calX$ and $\calY$ have twice the minimal volume among all complete orientable hyperbolic $4$-manifolds.
\end{abstract}

\section{Introduction}
By Margulis' Lemma, a finite-volume complete hyperbolic $n$-manifold $M^n$ has a finite number of ends called \emph{cusps}, each of which is homeomorphic to $N^{n-1} \times [0, +\infty)$ for some closed Euclidean $(n-1)$-manifold $N^{n-1}$. In dimension three many knot complements provide an example of a hyperbolic manifold with a single cusp \cite{Thurston}. In this setting the resulting manifolds are necessarily orientable, and since the torus $T^2$ is the only example of a closed Euclidean surface, the cusp is homeomorphic to $T^2 \times [0, +\infty)$.

In higher dimensions, where there is no direct analogue to geometrisation, constructing hyperbolic manifolds is a harder task. The goal is usually achieved by using either arithmetic methods, or producing direct constructions based on Poincar\'{e}'s fundamental polytope theorem.

It is also known that the cusp section of a single-cusped hyperbolic $4$-manifold cannot have any possible homeomorphism type of a closed Euclidean $3$-manifold \cite{LR}. In fact, a necessary condition for such a manifold to be the cusp section of a single-cusped hyperbolic $4$-manifold is that its $\eta$-invariant is an integer, and such a condition is met by only four out of the six possible orientable homeomorphism types.
The first example of a single-cusped hyperbolic $4$-manifold whose cusp section is a $3$-torus was given in \cite{KM}. To the best of our knowledge, no example of single cusped hyperbolic $4$-manifold with cusp section non-homeomorphic to the $3$-torus has been constructed before. In this paper, we provide an example of such an object (Theorem \ref{teo:nontoric}). The cusp section of our example is homeomorphic to the \emph{unique} Euclidean orientable $\mathbb{S}^1$-bundle over the Klein bottle $K$ (see \cite{Scott}).
 
It would also be interesting to generalise to dimension four the well-known $3$-dimensional examples of hyperbolic knot complements. A reasonable model would be the complement of a $2$-dimensional torus $T$ in $\mathbb{S}^4$. The boundary $\partial(U(T))$ of a regular neighbourhood $U(T)$ of $T$ will necessarily be homeomorphic to the $3$-torus $T^3$. Therefore $\partial(U(T))$ will admit a Euclidean structure, and we can hope to find a complete, finite volume hyperbolic structure on the complement of $S$ where $\partial(U(T))$ corresponds to a cross-section of the cusp.

However we have to be careful: since a manifold $M$ of dimension $n\geq4$ can admit a variety of non-equivalent smooth structures, there is a sharp difference in saying that $M$ is \emph{homeomorphic} or \emph{diffeomorphic} to some preferred model of a smooth manifold.
First examples of hyperbolic $4$-manifolds which are homeomorphic to complements of collections of $2$-dimensional tori in $\mathbb{S}^4$ were given in \cite{IRT}. All the examples constructed there have more than one cusp, and each of these cusps is homeomorphic to $T^3\times [0, +\infty)$, where $T^3$ denotes the three-dimensional torus. In \cite{S}, one of these manifolds (equipped with the smooth structure induced by its hyperbolic metric) is shown to be diffeomorphic to the complement of a collection of five tori in $\mathbb{S}^4$ with its standard smooth structure.

For a hyperbolic manifold $M$, being the complement of a collection of tori or Klein bottles in $\mathbb{S}^4$ implies strong restrictions on the geometry and topology of $M$ itself. A classical result due to H.C.~Wang states that, in all dimensions $n\geq 2$, the volumes of finite-volume hyperbolic $n$-manifolds form a well-ordered subset of $\mathbb{R}$, which is discrete if $n\neq 3$ \cite[Theorem 3.1]{Vinberg}. Moreover, in even dimensions, the volume of a hyperbolic $n$-manifold $\calM$ turns out to be proportional to the Euler characteristic $\chi(\calM)$ through the Gau\ss--Bonnet formula (\cite{Serre}, see also \cite{Heckman, Zehrt}) $$\text{Vol}\,M= c_n\cdot \chi(M),\;\,\,\text{with}\; c_n= (-2\pi)^{\frac{n}{2}}/(n-1)!!.$$ 

The Euler characteristic of a four-dimensional hyperbolic manifold $M$ (in general, non-compact, though finite-volume) can take any positive integer value \cite{RT}. However, if such a manifold is the complement of a collection of tori in $\mathbb{S}^4$, an easy excision property argument shows that necessarily $\chi(M)=\chi(\mathbb{S}^4)=2$, and therefore $\text{Vol}\;M=8\pi^2/3$.

In this setting, combining several properties of a manifold in a single example may become a difficult task (e.g. we want to construct a complete hyperbolic finite-volume $4$-manifold with a single cusp, and a small volume). In the present paper we use the technique of colourings (see \cite{DJ, GS, KMT, Vesnin87, Vesnin}) in order to produce examples of single-cusped hyperbolic manifolds of Euler characteristic $2$. The main results are the following:

\begin{teo}\label{teo:nontoric}
There exists an orientable, complete, finite-volume hyperbolic $4$-manifold $\calX$ such that $\chi(\calX)=2$ with a single cusp whose cross-section is homeomorphic to the Euclidean orientable $S^{1}$-bundle over the Klein bottle.
\end{teo} 

\begin{teo}\label{teo:toric}
There exists an orientable, complete, finite-volume hyperbolic $4$-manifold $\calY$ such that $\chi(\calY)=2$ with a single cusp whose cross-section is homeomorphic to a $3$-torus.
\end{teo} 

We point out the fact that the examples constructed in this work have minimal volume amongst known examples of orientable, single cusped, hyperbolic $4$-manifolds.
Finally, we point out that the techniques introduced in this paper allow us to produce many non-compact hyperbolic $4$-manifolds of Euler characteristic $2$. It is reasonable to wonder if some of these are new examples of complements of tori and/or Klein bottles in $\mathbb{S}^4$.

Also, our examples of hyperbolic manifolds with toric cusps may be used to produce new instances of Einstein manifolds, c.f. \cite{Bamler}, or integer homology spheres with non-vanishing simplicial volume, c.f. \cite{FM, RT2005}.

The paper is organised as follows: in Section \ref{section:colourings-general}, we recall basic properties of colourings and prove some auxiliary statements. Then, in Section \ref{section:colourings-cube}, we classify all possible colouring of a cube, up to homeomorphism of the resulting manifolds. In Section \ref{section:colourings-Potyagailo-Vinberg} we study the properties of colourings of a particular polytope: the $4$-dimensional Potyaga\u{\i}lo-Vinberg polytope first introduced in \cite{PV}. In Section \ref{section:symmetric-example} we produce a colouring of $P^4$ that we use in order to construct our examples. Finally, in Section \ref{section:mutations}, we use mutations along $3$-dimensional hyper-surfaces in $4$-dimensional multi-cusped hyperbolic $4$-manifolds to construct single-cusped hyperbolic $4$-manifolds with toric and non-toric cusp sections. 

\section{Properties of colourings}\label{section:colourings-general}

In this section we review the construction of manifolds from colourings of right-angled polytopes, with particular attention to the case of non-compact hyperbolic polytopes. We also discuss the relations between the algebraic properties of the colouring and the topological properties of the resulting manifold.

Let $\mathbb{X}^n$ with $\mathbb{X} = \mathbb{S}, \mathbb{E}$, or $\mathbb{H}$ denote, respectively, the $n$-sphere, $n$-dimensional Euclidean space, or $n$-dimensional hyperbolic (or Lobachevski\u{\i}) space. Let $P \subset \mathbb{X}^n$ be a convex polytope, and let $\calF(P)$ be the set of its co-dimension one faces (or \textit{facets}). A polytope $P \subset \mathbb{X}^n$ is called simple if each of its vertices belongs to exactly $n$ facets.

A colouring of a simple polytope $P$ according to \cite{DJ, GS, Vesnin87, Vesnin} is a map $\lambda: \calF(P) \rightarrow (\mathbb{Z}/2\mathbb{Z})^n$. A colouring is called \textit{proper} if the colours of facets around each vertex of $P$ are linearly independent vectors of $V = (\mathbb{Z}/2\mathbb{Z})^n$. 

One of most remarkable cases is a colouring of a compact right-angled polytope $P\subset \mathbb{H}^n$, with a proper colouring\footnote{by a result of L. Potyaga\u{\i}lo and \`E. Vinberg \cite{PV}, $2 \leq n \leq 4$}. Polytopes of this type give rise to interesting families of hyperbolic manifolds \cite{GS, KMT, Vesnin87, Vesnin}. 

In \cite{KMT} the notion of a colouring is extended to let $V$ be any finite-dimensional vector space over $\mathbb{Z}/2\mathbb{Z}$. Below we shall use the notion of colouring in a wider context of right-angled polytopes $P$, which are not simple. A polytope $P \subset \mathbb{X}^n$ is called \textit{simple at edges} if each edge belongs to exactly $(n-1)$ facets. In the case of a right-angled polytope $P \subset \mathbb{H}^4$, this means that vertex figures of $P$ are either $3$-dimensional tetrahedra or cubes \cite[Proposition~1]{D}. 

A colouring of a polytope $P\subset \mathbb{X}^n$ which is simple at edges is a map $\lambda: \calF(P) \rightarrow V$, where $V$ is a finite-dimensional vector space over $\mathbb{Z}/2\mathbb{Z}$. A colouring $\lambda$ is \textit{proper} if the following two conditions are satisfied:
\begin{enumerate}
\item \textit{Properness at vertices:} if $v$ is a simple vertex of $P$, then the $n$ colours of facets around it are linearly independent vectors of $V$;
\item \textit{Properness at edges:} if $e$ is an ideal edge (i.e. an edge having an ideal vertex, or both ideal vertices) of $P$, then the $(n-1)$ colours of facets around $e$ are linearly independent. 
\end{enumerate}

Let $P$ be a right-angled polytope in $\mathbb{X}^n$, with $\mathbb{X} = \mathbb{S}$, $\mathbb{E}$, or $\mathbb{H}$. In the latter case, we allow some (not necessarily all) of the vertices of $P$ to lie on the ideal boundary\footnote{by a result of G.~Dufour \cite{D} right-angled polytopes exist in $\mathbb{H}^n$ for $2 \leq n \leq 12$. Examples are known up to dimension $8$ \cite{PV}}. Let $\lambda:\calF(P)\rightarrow V$ be a colouring on the facets of $P$, with values in a finite-dimensional vector space $V$ over $\mathbb{Z}/2\mathbb{Z}$, which satisfies all the properness conditions.

The colouring defines a homomorphism from $G$ to $V$ which we continue to denote by $\lambda$, where $G=G(P)$ is the right-angled Coxeter group generated by reflections in the facets of $P$. The following key result holds:

\begin{prop}
Let $\lambda$ be a coloring of a right-angled polytope which satisfies all the properness conditions at hyperbolic vertices and at the ideal edges. The kernel of the associated homomorphism $\lambda:G\rightarrow V$ is a torsion-free subgroup of $G$. To this subgroup, there corresponds an orbifold cover 
$$\pi:\calM_{\lambda}\rightarrow P,$$ where $\calM_{\lambda}$ is a manifold.
\end{prop}

\begin{proof}
Consider the tessellation of $\mathbb{H}^n$ obtained by reflecting the polytope $P$ in its facets. The group $G$ acts on $\mathbb{H}^n$ preserving such tessellation. A torsion element in $G$ necessarily fixes a face $F$ of the tessellation. Up to conjugation, we can suppose that $F$ is a face of $P$. The stabilizer of $F$ is generated by the reflections in the facets of $P$ which contain $F$. By the properness condition, we see that the stabilizer of $F$ is mapped injectively in $V$ by $\lambda$ and the proposition follows.
\end{proof}

Let us suppose that the homomorphism $\lambda$ is surjective. The automorphism group of the covering is isomorphic to $V$. To visualise the automorphism, consider the manifold $\calM_{\lambda}$, together with its tessellation into copies of the polytope $P$.  Consider any copy of $P_1$ of $P$ in the tessellation of $\calM_{\lambda}$. The polytope $P_1$ will be adjacent to another copy $P_2$ of $P$ along a facet $F$. The automorphism corresponding to $\lambda(F)$ permutes $P_1$ and $P_2$ by reflection in their common facet $F$.

Let $\Gamma_\lambda$ be a rooted graph defined as follows:
\begin{itemize}
\item the root of $\Gamma_\lambda$ is a vertex $p$ corresponding to a copy of $P$, and labelled with $\lambda(p) = \mathbf{0} \in V$,
\item if two copies $P_1$ and $P_2$ of $P$ in the tessellation of $\calM_{\lambda}$ corresponding to the vertices $p_1$ and $p_2$ of $\Gamma_\lambda$ are adjacent through a facet $F$, then the edge $\{p_1, p_2\}$ is labelled $\lambda(F)$,
\item if two copies $P_1$ and $P_2$ of $P$ in the tessellation of $\calM_{\lambda}$ are adjacent through a facet $F$, then the labels $\lambda(p_1), \lambda(p_2) \in V$ of the respective vertices $p_1$ and $p_2$ of $\Gamma_\lambda$ satisfy $\lambda(p_1) + \lambda(p_2) = \lambda(F)$. 
\end{itemize}
We call $\Gamma_\lambda$ the developing graph of $\lambda$. If a vertex $p$ of $\Gamma_\lambda$ has label $v$, its root can be shifted into $p$ by adding the vector $v$ to all the vertex labels and declaring $p$ the new root. Thus, the choice of the root and the vertex labelling is not canonical, while the choice of the edge labels is canonical. Therefore, the vertices of the graph  $\Gamma_\lambda$, can be interpreted as points of an affine space $\calA$ over $V$.

Now, let $F$ be a facet of $P$. The facet $F$ lifts to a finite collection of totally geodesic hyper-surfaces in $\calM_{\lambda}$. To the facet $F$, there corresponds a subgroup $W(F)\subset V$ generated by $\lambda(F)$ and the colours $\lambda(F^\prime)$ of all the facets $F^\prime$ of $P$ adjacent to $F$. 

\begin{prop}\label{prop:colourings-general}
The following properties hold:
\begin{enumerate} 
\item[(1)] The number of totally geodesic hyper-surfaces in $\calM_{\lambda}$ corresponding to the facet $F$ is equal to the index of $W(F)$ in $V$.
\item[(2)] Cutting along \emph{all} the hyper-surfaces associated with the facets $F_1\dots,F_n$ which share a common colour $\lambda(F_1)=\dots=\lambda(F_n)$ separates $\calM_{\lambda}$ if and only if the colour $\lambda(F_i)$ does not lie in the subspace $U$ generated by all other colours.
\item[(3)] The hyper-surfaces associated with a facet $F$ are two-sided if and only if the colour $\lambda(F)$ does not lie in the subspace generated by the colours of all the facets adjacent to $F$.
\item[(4)] If two hyper-surfaces $\calH_1$ and $\calH_2$ are lifts of two non-intersecting facets of $P$, then they do not intersect in the manifold $\calM_{\lambda}$.
\item[(5)] The manifold $M_{\lambda}$ is orientable if and only if, for some isomorphism $V\cong(\mathbb{Z}/2\mathbb{Z})^s$, each colour $\lambda(F_i)$ has an odd amount of 1's.
\end{enumerate}
\end{prop}

\begin{proof}
\textit{(1)} The number of lifts of a facet $F$ is the order of the orbit of a single lift $\widetilde{F}$ of $F$ to $\calM_{\lambda}$ under the action of the deck transformation group. It equals the index of $\mathrm{Stab}(\widetilde{F})$, the stabiliser of $\widetilde{F}$ under this action. Obviously, $\mathrm{Stab(\widetilde{F})} = W(F)$. 

\medskip

\textit{(2)} 
Let $\overline{\calM_\lambda} := \calM_\lambda // \{ \pi^{-1}(F_1), \dots, \pi^{-1}(F_n) \}$ be the manifold $\calM_\lambda$ cut along \textit{all} the lift of $F_i$'s. 
Remove from the developing graph for the colouring $\lambda$ all the edges with label $\lambda(F_i)$ to obtain the developing graph $\overline{\Gamma_{\lambda}}$ for the manifold 
$\overline{\calM_\lambda}$. It is clear that the graph $\overline{\Gamma_{\lambda}}$ is connected \emph{if and only if} the colour $\lambda(F_i)$ lies in the subspace generated by all other colours, therefore the same property holds for the manifold $\overline{\calM_\lambda}$.

\medskip

\textit{(3)}
Let $F_1,\dots,F_k$ be the faces of $P$ adjacent to a given facet $F$. Let us choose as a root for the developing graph $\Gamma_{\lambda}$ a vertex $v$ adjacent to an edge with label $\lambda(F)$, and consider the connected sub-graph $\Gamma_{\lambda}^F\subset \Gamma_{\lambda}$ given by the vertices of $\Gamma_{\lambda}$ which can be joined to the root $v$ by edges with labels $\lambda(F)$, or $\lambda(F_i)$, $i=1\dots k$.

The graph $\Gamma_{\lambda}^F$ is the developing graph of a small tubular neighbourhood $\mathrm{Tube}(\calH_F)$ of a lift $\calH_F$ of a facet $F$ in $\calM_\lambda$, where each vertex corresponds to the intersection of a copy of the polytope $P$ with the neighbourhood $\mathrm{Tube}(\calH_F)$.  The hyper-surface $\calH_F$ is two-sided in $\calM_{\lambda}$ if and only if  $\mathrm{Tube}(\calH_F)\setminus \calH_F$ is disconnected. By applying an argument analogous to that of (2), we see that this holds if and only if $\lambda(F)$ does not lie in the subspace generated by the colours $\lambda(F_i)$, $i=1\dots k$.

\medskip

\textit{(4)} Consider the manifold $\calM_{\lambda}$, together with its tessellation into copies of the polytope $P$. The hyper-surfaces associated with a facet $F$ of $P$ are naturally tessellated by copies of the facet $F$. If two hyper-surfaces $\calH_1$ and $\calH_2$ intersect, they do so in the boundary of some copy of $P$, and therefore correspond to two facets $F_1$, $F_2$ of $P$ such that $F_1\cap F_2\neq \emptyset$. 

\medskip

\textit{(5)} See \cite[Lemma~2.4]{KMT}.

\end{proof}

\begin{defn}
A coloring $\lambda: \calF(P) \rightarrow V$ is \emph{orientable} if it satisfies condition $(5)$ of Proposition \ref{prop:colourings-general}.
\end{defn}

\subsection{Equivalent colourings and coloured isometries}\label{section:colourings-isometries}

Given a colouring $\lambda:G\rightarrow V$ (from now on, we shall systematically denote the colouring and the associated homomorphism by the same letter), there are a number of equivalent colourings, that give rise to manifolds isometric to $\calM_{\lambda}$. We can compose the homomorphism $\lambda:G\rightarrow V$ with any injective linear automorphism $\phi\in GL(V)$ in order to obtain a new colouring, which we denote by $\phi_{*}(\lambda)$. Formally speaking, we change each colour $\lambda(F)\in V$ associated with a facet $F$ of $P$ to the colour $\phi(\lambda(F))$. It is clear that the kernel of the homomorphism $\phi\circ \lambda$ coincides with the kernel of $\lambda$, and therefore the manifolds $\calM_{\lambda}$ and $\calM_{\phi_{*}(\lambda)}$ are isometric. However this operation changes the correspondence between the group of automorphisms of the colouring, which is generated by reflections in the facets of $P$, and the vector space $V$.

Notice that every symmetry of the polytope $P$ induces a natural permutation of its facet colours.
\begin{defn}
Let $\lambda:G(P)\rightarrow V$ be a colouring. A symmetry $\psi$ of $P$ is \emph{admissible with respect to the colouring $\lambda$} if the naturally associated map on the facet colours is realised by a linear automorphism $\phi\in GL(V)$ of $V$. Admissible symmetries form a subgroup $\text{Adm}_{\lambda}(P)$ of $\text{Symm}(P)$, and there is a naturally defined homomorphism from $\text{Adm}_{\lambda}(P)$ to $GL(V)$. We denote by $GL(\lambda)$ the image of this homomorphism.
\end{defn}

Now let $\text{Adm}_{\lambda}(P) < \text{Symm}(P)$ be the group of admissible symmetries of $P$ with respect to a colouring $\lambda:G(P)\rightarrow V$. Each symmetry $\psi \in \text{Adm}_{\lambda}(P)$ lifts to an isometry of $\calM_{\lambda}$, and obviously any two lifts differ by composition with an element of $\text{Aut}\, \pi\cong V$. 

\begin{defn}
The group of lifts of admissible symmetries of $P$ to $\calM_{\lambda}$ is called the \emph{coloured isometry group} of $\calM_{\lambda}$, and is denoted by $\text{Isom}_c(\calM_{\lambda})$.
\end{defn}

There is a short exact sequence 
$$0\rightarrow \text{Aut}\, \pi \cong V\rightarrow \text{Isom}_c(\calM_{\lambda})\rightarrow \text{Adm}_{\lambda}(P)\rightarrow 0. $$

This sequence naturally splits: it is sufficient to choose a copy $P_0$ of $P$ in the tessellation of $\calM_{\lambda}$, and lift every $\psi \in \text{Adm}_\lambda(P)$ to the \emph{unique} element of $\text{Isom}_c(\calM_{\lambda})$ which maps $P_0$ to itself. Therefore we write

$$\text{Isom}_c(\calM_{\lambda})\cong V \rtimes \text{Adm}_{\lambda}(P),$$ where the action of  $\text{Adm}_\lambda(P)$ on $V$ is induced by the natural action of 
$GL(\lambda)$ on $V$.

Therefore, given a colouring $\lambda:\calF(P)\rightarrow V$, the group $\text{Isom}_c(\calM_{\lambda})$ acts on the developing graph $\Gamma_{\lambda}$ through affine isomorphisms of its vertex set $\calA$. The elements of $\text{Aut}\, \pi$ correspond to translations in $\calA$, while elements of $\text{Adm}_\lambda(P)$ induce linear maps in $GL(\lambda)$ of the underlying vector space $V$. 

Moreover, to each stratum $S$ in the cell decomposition of the manifold $\calM_{\lambda}$, there corresponds an affine subspace $\mathcal{W}_S\subset \calA$. Notice that each such $n$-stratum $S$ corresponds to a \emph{unique} $n$-face $F(S)$ of the polytope $P$. The elements of $\mathcal{W}_S$ correspond to the copies of the polytope $P$ which are adjacent to the stratum $S$. These form an affine subspace of $\cal{A}$ whose \emph{dimension} is equal to the \emph{co-dimension} of $S$. The underlying linear subspace of $V$ is trivial in the case of strata of co-dimension $0$ (copies of the polytope $P$), and, for strata corresponding to a co-dimension one face $F$, it coincides with the subspace generated by the colour $\lambda(F)$. In all other cases, the underlying linear subspace coincides with the subspace generated by the colours of the facets of $P$ adjacent to $F(S)$.

\begin{rem}
In order to check if an element $\psi\in\text{Isom}_c(\calM_{\lambda})$ acts on the manifold $\calM_{\lambda}$ without fixed points, it is sufficient to check if no affine subspace of the form $\mathcal{W}_S\subset \calA$ is preserved by $\psi$.
\end{rem}

\begin{rem}\label{remark:affine-subspaces}
Notice that the discussion above also applies to the cusps of $\calM_{\lambda}$. Each of these will be associated with an ideal vertex $v$ of the polytope $P$, and with each cusp we can associate an affine subspace of $\cal{A}$, generated by the colours carried over to the vertex figure of $v$ from the respective facets of $P$ adjacent to $v$ (i.e. a \textit{restriction} of $\lambda$ to the vertex figure of $v$).
\end{rem}

\section{Colourings of a three-dimensional cube}\label{section:colourings-cube}
In this section we discuss an example of colouring on the faces of a $3$-dimensional Euclidean cube, in order to show how to compute the topology and the geometry of the manifolds associated with such colourings. Moreover, we derive a criterion to determine the topology of a Euclidean manifold $\calM_{\lambda}$ associated with a colouring $\lambda$ on the faces of a cube in terms of the intersections of the vector spaces generated by the colours on opposite faces.

\begin{example}\label{ex:nontoric}
Here we exhibit an example of colouring $\lambda$ on the faces of a unit cube $C$ which produces a non-toric flat manifold $\mathcal{M}_\lambda$, namely, the one which which fibres over $S^1$ with a torus as fibre and monodromy given by a hyper-elliptic involution. This manifold can be described also as the Euclidean orientable $\mathbb{S}^1$-bundle over the Klein bottle $K$, and it is of topological type $\calG_2$ in Wolf's census of closed Euclidean $3$-manifolds \cite{Wolf}.

Consider a Euclidean unit cube with its three pairs of opposite square faces. We assign colours from $V=(\mathbb{Z}/2\mathbb{Z})^3$ to them in the following way (see Figure~\ref{fig:cube-colouring-example1}):
\begin{enumerate}
\item Assign the colour $(1,0,0)$ to a pair of opposite faces,
\item Assign the colour $(0,1,0)$ to the second pair of opposite faces,
\item Assign the colours $(0,0,1)$ and $(1,1,1)$ to the remaining two faces.
\end{enumerate}

\begin{figure}[h]
\centering
\includegraphics[scale=0.4]{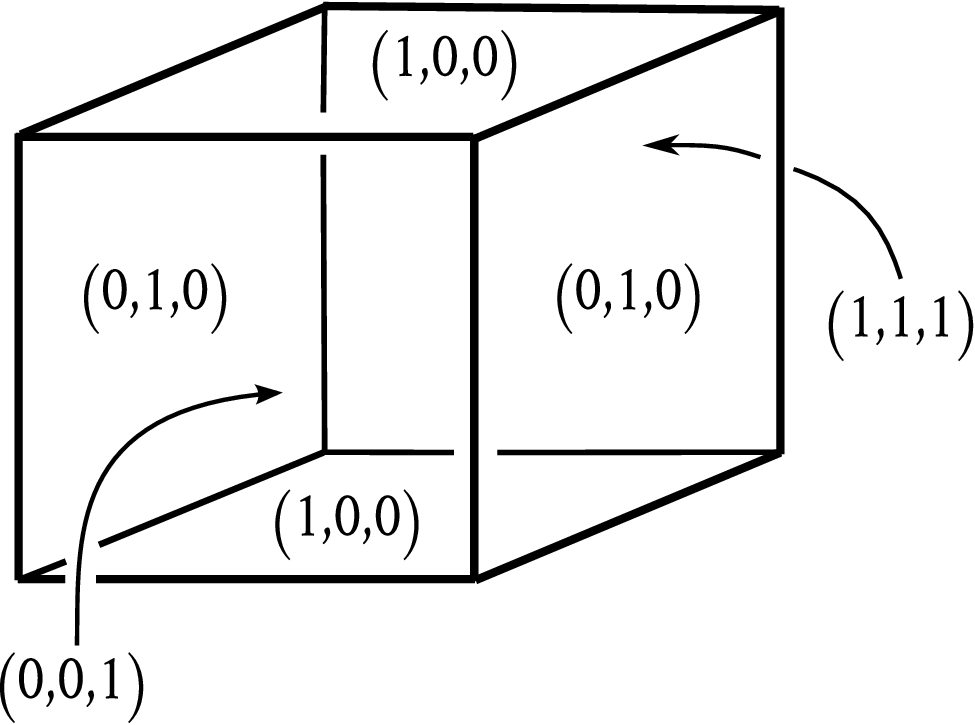}
\caption{The colouring $\lambda$ of the cube $C$, faces are labelled with vectors from $V$}
\label{fig:cube-colouring-example1}
\end{figure}

The third pair of opposite faces is parallel to the fibres of the fibration. We may check that the monodromy is a hyper-elliptic involution by performing the following computation. 

Let us start with a cube $C_0 = C$, labelled with $\mathbf{0} \in V$ and keep reflecting in its faces. We also translate the face colours from each copy of $C$ to its reflected counterpart accordingly. If two copies of $C$, say $C_1$ and $C_2$, share a face $S$, then their labels sum up to $\lambda(S)$. Once we reach another cube $C_N$ labelled $\mathbf{0}$, we determine an isometry $\psi \in \text{Isom}\,\mathbb{E}^3$ that brings $C_N$ back to $C_0$ respecting their face colourings. This isometry is exactly the composition of the chosen reflections, and it belongs to the group of automorphisms of the covering $\mathbb{E}^3\rightarrow \calM_{\lambda}$, since $C_0$ and $C_N$ correspond to the same vertex in the developing graph $\Gamma_\lambda$. 

We repeat this process until we are able to find three linearly independent isometries $\psi_i \in \text{Isom}\,\mathbb{E}^3$, $i=1,2,3$, which generate a manifold of the correct volume, namely the order, as a set, of the image of the colouring $\lambda$ in the vector space $V$, since these guarantees that the isometries $\psi_i$, $i=1,2,3$, generate the automorphism group of the covering $\mathbb{E}^3\rightarrow \calM_{\lambda}$.

\begin{figure}[h]
\centering
\includegraphics[scale=0.4]{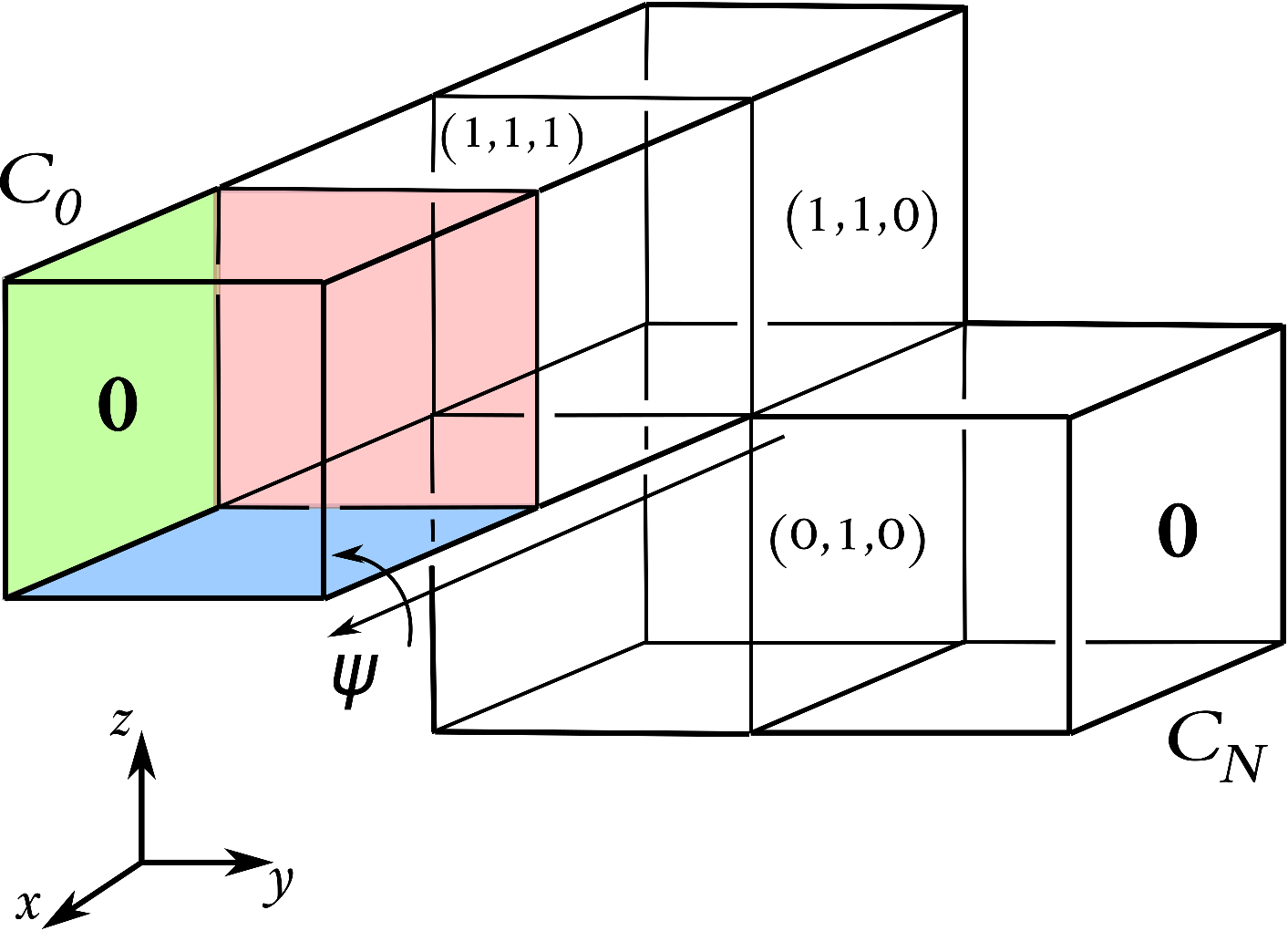}
\caption{Moving through copies of $C$ by reflections in their faces, from $C_0$ to $C_N$. The isometry $\psi$ is a translation + half-turn rotation that brings us back from $C_N$ to $C_0$}
\label{fig:cube-colouring-example2}
\end{figure}

In the case of the above described colouring of a cube (see Figure~\ref{fig:cube-colouring-example1}) we apply this process to obtain a sequence of five cubes depicted in Figure~\ref{fig:cube-colouring-example2}, where the first $C_0$ and the last $C_N = C_4$ cubes are labelled $\mathbf{0}$ and the isometry identifying them is 
\begin{equation*}
\psi_1((x,y,z)^T) = \left(\begin{array}{ccc}
 1& 0& 0\\
0& -1& 0\\
0& 0& -1
\end{array}\right)\, \left(\begin{array}{c}
x\\
y\\
z\\
\end{array}\right) +
\left(\begin{array}{c}
2\\
0\\
0
\end{array}\right)
\end{equation*}
Here, we start by reflecting the cubes in the $(1,1,1)$--face. Analogously, we obtain the respective isometries $\psi_2$ and $\psi_3$ by reflecting first in the $(0,1,0)$-- or $(1,0,0)$--face. We get
\begin{equation*}
\psi_2((x, y, z)^T) = (x, y + 2, z)^T,\,\,\, \psi_3((x, y, z)^T) = (x, y, z + 2)^T.
\end{equation*}
Then, the translation lattice of $\mathcal{M}_\lambda$ is generated by $\psi_i$'s, and its volume equals $\mathrm{Vol}\,\mathcal{M}_\lambda = 8$. Thus, the manifold $\mathcal{M}_\lambda$ has type $\mathcal{G}_2$ according to \cite{Wolf}. 
\end{example}

The geometry and topology of the manifold $\calM_{\lambda}$ are determined only by the equivalence class of the colouring ${\lambda}$. Therefore, understanding the topology of the cusp sections starting from a colouring of a cube, depends only on the equivalence class of the colouring. This allows us to describe the topology of the resulting cusp shape by understanding the intersections between the different subspaces generated by the colours of pairs of opposite faces of that cube.

Let us denote by $C$ the unit cube in the Euclidean $3$-space $\mathbb{E}^3=\{(x,y,z) \, | \, x,y,$ $z\in \mathbb{R}\}$, and let $G$ be the associated reflection group. The group $G$ is generated by reflections in the planes of the form $x=c,y=c,z=c$, with $c\in \mathbb{N}$, and is a right-angled Coxeter group. Reflections in parallel planes form three subgroups $G_x, G_y, G_z\subset G$.

An orientable, proper $V$-colouring $\lambda$ of the square faces of a unit cube $C$ defines three linear subspaces of $V$, that we call $V_x$, $V_y$, $V_z\subset V$, respectively, each generated by the colours assigned to a pair of opposite facets of $C$. The colouring also defines three homomorphisms $\lambda_x:G_x\rightarrow V_x$, $\lambda_y:G_y\rightarrow V_y$ and $\lambda_z:G_z\rightarrow V_z$.

In order to understand the topology of the manifold $\calM_{\lambda}$ obtained from the colouring $\lambda$ of the cube $C$, we shall consider the \textit{restriction} of $\lambda$ to a section of $C$ by the plane $x = c$, $0 < c < 1$. This section is a unit square $S$, whose edges inherit their colours from the respective faces of $C$. Indeed, each edge of $S$ arises at the intersection of the plane $x = c$ with a face $F$ of $C$, orthogonal to $x = c$. 

The restriction of $\lambda$ to $S$ is an orientable colouring, with values in the vector space of the form $V_y+V_z$, and, by the orientability assumption, it produces a torus $T_x$ with volume $2^{\,\text{dim}(V_y+V_z)}$. The vector space $V_y+V_z$ acts on the torus $T_x$ by automorphisms of the covering $T_x\rightarrow S$. Clearly, all sections of $C$ by the planes $x = c$, $0 < c < 1$, lift to tori in $\calM_\lambda$, which constitute the fibres of the fibration of $\calM_\lambda$ over the circle $\mathbb{S}^1$. 

We note, that in the exposition below the variables $x$, $y$ and $z$ are interchangeable: no direction along which we may fibre $\calM_\lambda$ is preferred. The intersection properties of the subspaces $V_x$, $V_y$ and $V_z$  that we are interested in remain intact under a permutation of $x$, $y$ and $z$. 

We shall determine the monodromy map $T_x \rightarrow T_x$ of the torus fibration over $\mathbb{S}^1$ that represents $\calM_\lambda$. Any orientation-preserving automorphism of $T_x$ is expressible as a sum of an \emph{even} number of face colours assigned to the faces parallel to the planes $y=0$ and $z=0$, and is necessarily isotopic to the identity or to a hyper-elliptic involution. The elements of $V_y+V_z$ that induce maps isotopic to the identity form a subgroup of $V_y+V_z$, which has index two in the group of orientation-preserving isometries of $T_x$. It consists of elements which are expressed as a sum of an even number of face colours in \emph{both} $V_y$ and $V_z$.

The torus $T_x$ is obtained as a quotient of $\mathbb{E}^2$ by translations along certain two vectors $v_1$ and $v_2$. By embedding $\mathbb{E}^2$ into $\mathbb{E}^3=\mathbb{R}\times\mathbb{E}^2$, we see that the vectors $(0,v_1),(0,v_2)$ belong to the translation lattice associated with the Euclidean $3$-manifold $\calM_{\lambda}$. We can add a third translation in $G_x$ along a vector of the form $v_3=(n,0)$, generated by reflections in the square faces parallel to $x=0$, in order to define a lattice $L$. The volume of $\mathbb{E}^3/L$ equals $$2^{\,\text{dim}\,(V_x)+\text{dim}\,(V_y+V_z)}.$$  

The length of the vector $v_3$ equals $2$ if the subspace $V_x$ is one-dimensional (in other words, if the colours of the respective pair of opposite facets are the same), and it equals $4$ if $V_x$ has dimension $2$ (in this case, the colours of opposite faces corresponding to $V_x$ are distinct). Notice, that the lattice defined above \emph{does not} necessarily coincide with the maximal translation lattice associated with the manifold $\calM_{\lambda}$.

If the intersection $V_x\cap(V_y+V_z)$ is trivial, the lattice $L$ is indeed the full group of isometries which defines the cover $\mathbb{E}^3\rightarrow \calM_{\lambda}$. To see this, it suffices to notice that $$\text{dim}\,V=\text{dim}\, (V_x+V_y+V_z)= \text{dim}(V_x)+\text{dim}(V_y+V_z).$$ This means that the torus defined by the lattice $L$ has the same volume as the manifold $\calM_{\lambda}$, and therefore the covering $\mathbb{E}^3/L\rightarrow \calM_{\lambda}$ is trivial.

If the intersection $V_x\cap (V_y+V_z)$ is non-trivial, we have to consider all possible cases separately.

\begin{enumerate}
\item[(1)] Suppose that $V_x$ has dimension $1$, and pick the unique vector $w\in  V_x\cap (V_y+V_z)$. Clearly, the reflection in any of the two square faces of the cube $C$ parallel to $x=0$ gives an isometry $\psi\in G_x$ of $\mathbb{E}^3$ such that $\lambda_x(\psi)=w$.

Moreover, since $w\in V_y + V_z = \text{Aut}(T_x\rightarrow S)$ we can find a (necessarily orientation-reversing) isometry $\phi$ of $\mathbb{E}^3$ (belonging to the group generated by $G_y$ and $G_z$), such that $\lambda_x(\phi)=w$. The isometry $\phi$ induces an index two covering $T_x\rightarrow K$, where $K$ is a Klein bottle. Moreover $$\phi\circ \psi\in \text{ker}\,\lambda=\pi_1(\calM_{\lambda}).$$

The quotient of $\mathbb{E}^3/L$ under this isometry is the manifold $\calM_{\lambda}$, since 
$$\text{dim}\,V=\text{dim}\, (V_x+V_y+V_z)= \text{dim}(V_x)+\text{dim}(V_y+V_z)-1$$ and $\text{Vol}\,  (\mathbb{E}^3/L)=2\cdot \text{Vol}\,( \calM_{\lambda})$.  The resulting manifold can be described as the orientable Euclidean $\mathbb{S}^1$-bundle over the Klein bottle $K$, which is also a torus bundle over $\mathbb{S}^1$, with a hyper-elliptic involution acting on its fibres.

\item[(2)] Now, suppose that $V_x$ has dimension $2$ (\textit{i.e.} the colours of the faces of $C$ parallel to $x=0$ are different). There are two possible cases:

\begin{itemize}
\item[(2.a)] If $V_x\cap (V_y+V_z)$ has dimension $1$, it contains a unique non-zero vector $w$. If $w$ is one of the face colours which generate $V_x$, then we can proceed as in the previous case and verify that the resulting manifold is a torus bundle over the circle, with the hyper-elliptic involution as monodromy map. 

If $w$ is the sum of the face colours which generate $V_x$, then the isometry $\psi\in G_x$ given by the translation along the vector $(0,0,2)$ (which is generated by reflections in the faces of $C$ that are parallel to $x=0$), satisfies $\lambda_x(\psi)=w$. As before, we can find an orientation-preserving isometry $\phi\in \langle G_y,G_z \rangle$ of $\mathbb{E}^3$ such that $\lambda_x(\phi)=w$. We need to check whether $\phi$ induces the identity or a hyper-elliptic involution on the torus $T_x$, \textit{i.e.} whether the vector $w$ is the sum of an even or an odd number of the face colours in $V_x$ and $V_z$. By proceeding as in the previous steps, we see that in the first case the resulting manifold is a torus. Otherwise, it is a torus bundle with a hyper-elliptic involution as monodromy. 

\item[(2.b)] If $V_x\subset (V_y+V_z)$, the resulting manifold is non-toric. We deduce, in analogy to (1), that it is a torus bundle with a hyper-elliptic involution as monodromy.
\end{itemize}
\end{enumerate}

We summarise the discussion above in the following way:

\begin{prop}\label{prop:cube-colourings-mflds}
Let $\lambda$ be a proper, orientable colouring on the faces of a Euclidean cube. The Euclidean manifold $\calM_{\lambda}$ is either a $3$-torus, or a torus bundle over $\mathbb{S}^1$ with monodromy given by a hyper-elliptic involution.

The manifold is a $3$-torus if and only if one of the following two conditions in fulfilled:

\begin{enumerate}
\item One of the vector spaces $V_x, V_y, V_z\subset V$ has trivial intersection with the sum of the other two.
\item One of the vector spaces $V_x, V_y, V_z\subset V$ (without loss of generality, we can suppose that it is $V_x$), has a one-dimensional intersection with the sum of the other two, and the unique non-trivial vector $w=w_y+w_z$, $w_y\in V_y, w_z\in V_z$, in this intersection is expressed as the sum of the two face colours that generate $V_x$, while the vectors $w_y$ and $w_z$ are expressed as the sum of an even number of each of the face colours that generate $V_y$ and $V_z$, respectively. 
\end{enumerate}

\end{prop}

\section{Colourings of the polytope $P^4$}\label{section:colourings-Potyagailo-Vinberg}

Here we shall discuss general properties of colourings of the right-angled hyperbolic polytope $P^4$ introduced by L. Potyaga\u{\i}lo and \`E. Vinberg in \cite{PV}, discuss an example and use it to build single cusped manifolds whose cusp section is either a $3$-torus, or an orientable $\mathbb{S}^1$-bundle over the Klein bottle. 

We begin by describing the combinatorial dual of the polytope $P^4$: the \emph{rectified $5$-cell}, which we denote by $R$. The polytope $R$ is obtained by considering the convex hull in $\mathbb{R}^5$ of the set of ten points obtained by permuting the coordinates of the vector $(1,1,1,0,0)$. Alternatively, one can obtain the rectified $5$-cell starting from a $4$-dimensional Euclidean simplex $S_4$, and considering the convex hull of the midpoints of its edges.

As a consequence of the latter construction for $R$, we see that there is the following sequence of one-to-one correspondences:
\begin{equation}
\{\text{Facets of}\; P^4\}\leftrightarrow \{\text{Vertices of}\; R\}\leftrightarrow\{\text{Edges of}\; S_4\} 
\end{equation}
and defining a colouring on the facets of $P^4$ is equivalent to defining a colouring on the edges of a $4$-simplex $S_4$, \textit{i.e.}\ an edge colouring of the complete graph $K_5$. 

The polytope $P^4$ has a hyperbolic realisation as a right-angled polytope with ten facets which are copies of the polyhedron $P^3$ \cite{PV}, depicted in Figure~\ref{fig:facets}. Each copy is labelled as a vertex of the rectified $5$-cell, \textit{i.e.} by a permutation of the coordinates of the vector $(1,1,1,0,0)$.
Moreover, the polytope $P^4$ has $10$ vertices in total. Of these, $5$ are ideal and the other $5$ are compact and lie inside the hyperbolic space. By marking the vertices of $S_4$ with the integers in the set $\{1,2,3,4,5\}$,  each ideal (resp. compact) vertex of $P^4$ receives a natural label: opposite to each ideal vertex, there is a compact vertex with the same label. To an edge of $S_4$ connecting vertices with labels $i$ and $j$, there corresponds a facet of $P^4$ whose label has $0$'s in the $i$-th and $j$-th entries, and $1$'s in all other entries. For instance, to the edge connecting the vertices $4$ and $5$ of $S_4$, there corresponds a facet of $P^4$ with label $(1,1,1,0,0)$.

\begin{figure}[ht]
\centering
\includegraphics[width=0.35\textwidth]{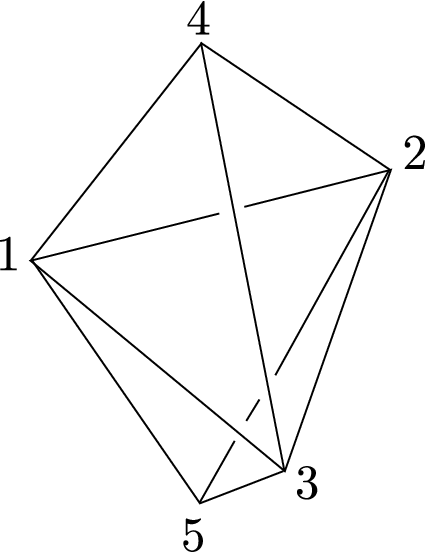}
\caption{The $P^3$-facet of $P^4$ labelled $(1,1,1,0,0)$, together with labels on its vertices. The vertices $1,2,3$ are ideal, while $4$ and $5$ lie in hyperbolic space.}\label{fig:facets}
\end{figure}

The vertex figure of an ideal vertex is a Euclidean cube. The hyperbolic volume of $P^4$ is equal to $(1/16)\cdot v_m$ \cite{ERT}, where $v_m=4\pi^2/3$ is the minimal volume of a hyperbolic $4$-manifold.

Observe that each facet of $P^4$ corresponds to an edge of the complete graph $K_5$, and that two edges of $K_5$ share a vertex if and only if the corresponding facets of $P^4$ are incident. Given an edge $e$ of $K_5$, there are exactly $3$ other edges which do not share a vertex with $e$, and these edges form a complete sub-graph $K_3\subset K_5$. As a consequence we see that, given a facet $F$ of $P^4$, there are exactly $3$ other (pairwise intersecting) facets of $P^4$ which are disjoint from $F$. 

In order to obtain a manifold, we need to make sure that the colouring of $P^4$ is proper at both the compact vertices and the ideal edges. This is equivalent to the following two conditions imposed on the edge colouring of the graph $K_5$:

\begin{enumerate}
\item \emph{Properness at vertices}: The colours associated with the edges emanating from a vertex $v$ of $K_5$ are linearly independent vectors. There are $5$ vertices to check.
\item \emph{Properness at edges}: The colours associated with every complete sub-graph $K_3 \subset K_5$ (built by removing any two vertices and the edges incident to these vertices) are linearly independent vectors. There are $10$ such sub-graphs to check. 
\end{enumerate}

\subsection{Induced colourings}
Now let $\mathcal{F}(P^4)$ be the set of facets of $P^4$, and $V$ a finite dimensional vector space over $\mathbb{Z}/2\mathbb{Z}$. Let us call $M_{\lambda}$ the manifold associated with a colouring $\lambda:\mathcal{F}(P^4)\rightarrow V$ satisfying the properness conditions. Notice that, by the condition on the volume of $P^4$, the dimension of $V$ has to be at least $4$. We want to study the induced colourings. 

For each edge $e$ of $K_5$, we have a sub-graph of $K_5$ (called a \emph{dart}) formed by the edge $e$ and all other edges which share a vertex with $e$. It has 7 edges in total, 2 vertices of valence 4, and 3 vertices of valence 2, as shown in Figure~\ref{fig:dart}. The edges of the dart associated with an edge $e$ correspond to the facets of $P_4$ which are adjacent to the facet associated with $e$.

Given an edge $e$ of $K_5$, the colouring on its corresponding dart graph naturally induces a $V_F$-colouring on the $P^3$-facet $F$ of $P^4$ corresponding to $e$, where $V_F=V/\langle \lambda(F) \rangle$, and therefore defines an embedded, totally geodesic sub-manifold of the manifold $M_{\lambda}$.

\begin{figure}[ht]
\centering
\includegraphics[width=0.27\textwidth]{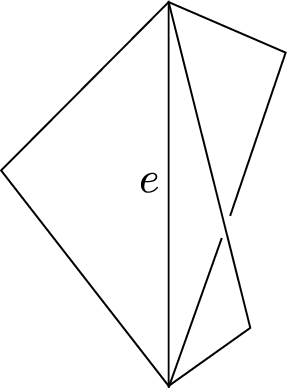}
\caption{The dart graph associated with an edge $e$ of $K_5$.}\label{fig:dart}
\end{figure}

Each ideal vertex $w$ of $P^4$ naturally corresponds to a sub-graph $K_4 \subset K_5$. The colouring on the edges of $K_4$, which correspond to the facets of $P^4$ that meet at $w$, naturally induces a colouring on the cubical vertex figure of $w$, with values in the vector space $V_w\subset V$ generated by the colours of the facets adjacent to $w$. The index of $V_w$ in $V$ determines the number of cusps that correspond to the vertex $w$.

\section{A highly symmetric example}\label{section:symmetric-example}

\subsection{Defining the colouring}\label{subsection:colouring-definition}

Consider a colouring of the complete graph $K_5$ with values in $V=(\mathbb{Z}/2\mathbb{Z})^5$ obtained in the following way:
\begin{enumerate}
\item Label the vertices of $K_5$ with the numbers $\{1,2,3,4,5\}$.
\item Assign to the edge connecting vertices $i$ and $j$ the element of $V$ which has $0$'s in the entries $i$ and $j$, and $1$'s in all other entries. For instance, the edge connecting vertices $1$ and $2$ will be coloured with $(0,0,1,1,1)$. 
\end{enumerate}

This edge colouring of $K_5$ defines a facet colouring $\lambda$ of the polytope $P^4$. It is easy to check that the properness conditions are satisfied, and therefore this colouring defines an orientable hyperbolic four-manifold $\mathcal{M} := \mathcal{M}_\lambda$, tessellated by $2^5 = 32$ copies of the polytope $P^4$. 

The group $\text{Adm}_\lambda(P^4)$ of admissible symmetries of $P^4$ coincides with its full isometry group $\text{Symm}(P^4)\cong \mathfrak{S}_5$, where the action on $V$ is given by permutation of the coordinates. Therefore we have a short exact sequence

\begin{equation} 
0\rightarrow V \rightarrow \mathrm{Isom}_c(\mathcal{M}) \rightarrow \mathfrak{S}_5\rightarrow 0 
\end{equation}\label{symmetries}

\subsection{Identifying hyper-surfaces}\label{subsection:hypersurfaces}

By applying Proposition \ref{prop:colourings-general}, we see that:

\begin{rem}
To each facet $F$ of $P^4$ there corresponds a \emph{unique}, non-separating, two-sided totally geodesic hyper-surface $\calH_{F}$ in $\mathcal{M}$.
\end{rem}

Moreover, since  the group of admissible symmetries of the colouring acts transitively on the facets of $P^4$, the hyper-surfaces  defined by the induced colourings on the facets are isometric to each other.

The colouring of the dart graphs associated with the edges having colours $(1,1,1,0,0)$ and $(0,0,1,1,1)$ are represented in Figure~\ref{fig:coloreddarts}.
\begin{figure}[ht]
\centering
\subfigure{
\includegraphics[width=0.35\textwidth]{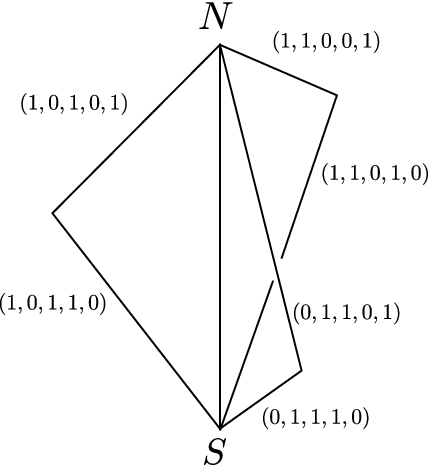}
}
\subfigure{
\includegraphics[width=0.35\textwidth]{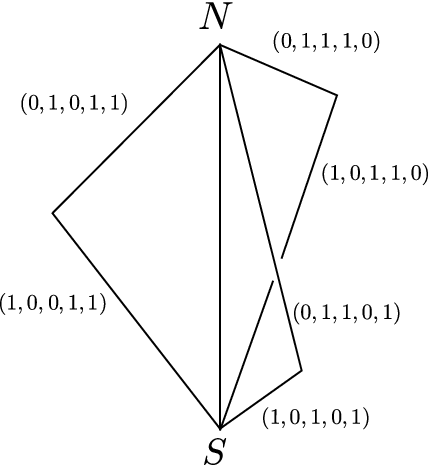}
}
\caption{The induced colouring on the dart graphs associated with the edges $(1,1,1,0,0)$ (left) and $(0,0,1,1,1)$ (right).}\label{fig:coloreddarts}
\end{figure}
The induced colouring $\lambda$ on a facet $F$ takes its values in the vector space $V_F = V / \langle \lambda(F)\rangle$, which is a vector space over $\mathbb{Z}/2\mathbb{Z}$ of dimension $4$. A basis for this vector space is given by the colours associated with the three edges which have a common endpoint at the vertex $S$ (resp.\ $N$), together with the colour associated with any other edge which has $N$ (resp.\ $S$) as a vertex, for instance $(1,1,0,1,0)$, $(1,0,1,1,0)$, $(0,1,1,1,0)$ and $(0,1,1,0,1)$.

Let us denote by $\mathcal{H}$ the hyperbolic $3$-manifold defined by the induced colouring $\lambda$ on the facets of $P^4$. Once again, the group of admissible symmetries $\text{Adm}_\lambda(F)$ coincides with the full symmetry group $\text{Symm}(F)\cong \mathbb{Z}/2\mathbb{Z}\times \mathfrak{S}_3$, and the action is given by permutation of the coordinates. 

We have a short exact sequence 

\begin{equation} 
0\rightarrow V' \rightarrow \mathrm{Isom}_c(\mathcal{H}) \rightarrow \mathrm{Symm}(F)\rightarrow 0,
\end{equation}\label{hypersurfaces}

where $\mathrm{Symm}(F)\cong \mathbb{Z}/2\mathbb{Z}\times \mathfrak{S}_3$ is the symmetry group of $F$. 
 


\subsection{Cusps}\label{subsection:cusps-mutations}

By a symmetry argument, it is clear that all the cusps of the manifold $\mathcal{M}$ are isometric to each other. The colours on each complete sub-graph $K_4\subset K_5$ generate a $4$-dimensional subspace $V''$ of $V=(\mathbb{Z}/2\mathbb{Z})^5$. Therefore, to each ideal vertex $v$ of $P^4$ there correspond \emph{two} cusps  of the manifold $\mathcal{M}$, for a total of ten cusps. Each cusp is tessellated by $2^4=16$ copies of a unit Euclidean cube.

Recall that the colours on the vertex figure of $P^4$ correspond to the colours associated with a complete sub-graph $K_4\subset K_5$. A basis for $V''$ is given by the colours associated with any set of three edges spanning a complete sub-graph $K_3\subset K_4$, together with the colour associated with one of the remaining edges. The choice of such a basis allows us to identify $V''$ with $(\mathbb{Z}/2\mathbb{Z})^4$. 

In Figure~\ref{fig:cuspcoloring}, we choose such a basis and represent the induced colouring directly on the cubical vertex figure of $P^4$. 

\begin{figure}[ht]
\centering
\includegraphics[width=0.6\textwidth]{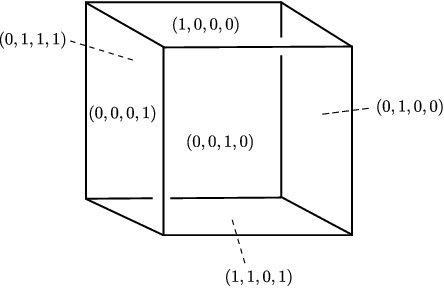}
\caption{The induced colouring on the cubical vertex figure of $P^4$, up to a suitable choice of a basis for the vector space $V'' \cong (\mathbb{Z}/2\mathbb{Z})^4$.}\label{fig:cuspcoloring}
\end{figure}

By applying Proposition \ref{prop:cube-colourings-mflds}, we conclude that the resulting cusps are all $3$-dimensional tori. An explicit computation, analogous to that of Example \ref{prop:cube-colourings-mflds} shows that the associated translation lattice is generated by translations along the vectors $(2,2,0)$, $(2,0,2)$ and $(0,2,2)$. 

To conclude, we discuss the cusp shapes of the hyper-surfaces associated with the facets of $P^4$, and how these intersect the cusps of $\mathcal{M}$.
Recall that each $P^3$-facet $F$ of the polytope $P^4$ has an induced colouring in the $4$-dimensional vector space $V'=V/\langle \lambda(F) \rangle$, where $\lambda(F)$ is the colour assigned to the facet $F$. The vertex figures of the polyhedron $P^3$ are unit Euclidean squares. It turns out that the induced $V'$--colouring on the edges of each square generate a subspace $V'''$ of $V'$ of dimension $3$. A basis for $V'''$ is given by the colours $c_1$, $c_2$, $c_3$ assigned to any three edges of $K_4 \subset K_5$ representing $F$, and the colour of the remaining edge is given by $c_1+c_2+c_3$, as in Figure~\ref{fig:squarecolors}. 

\begin{figure}[ht]
\centering
\includegraphics[width=0.5\textwidth]{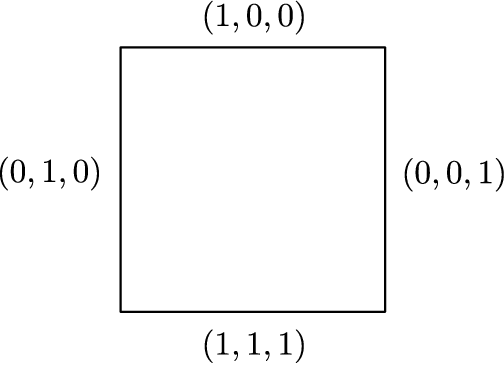}
\caption{The induced colouring on the square vertex figure of $P^3$, up to a suitable choice of a basis for the vector space $V''' \cong (\mathbb{Z}/2\mathbb{Z})^3$.}\label{fig:squarecolors}
\end{figure}

Again, to each ideal vertex of a facet $F\cong P^3$ there correspond two cusps in the hyper-surface $\mathcal{H}_F$ associated with $F$, for a total of six cusps. Each cusp is tessellated by $8$ unit squares. The Euclidean cusp shape is that of a flat torus $T$, generated by translations along the vectors $(2,2)$ and $(2,-2)$ represented in Figure~\ref{fig:cuspshape}.

\begin{figure}[htbp]
\centering
\includegraphics[width=0.45\textwidth]{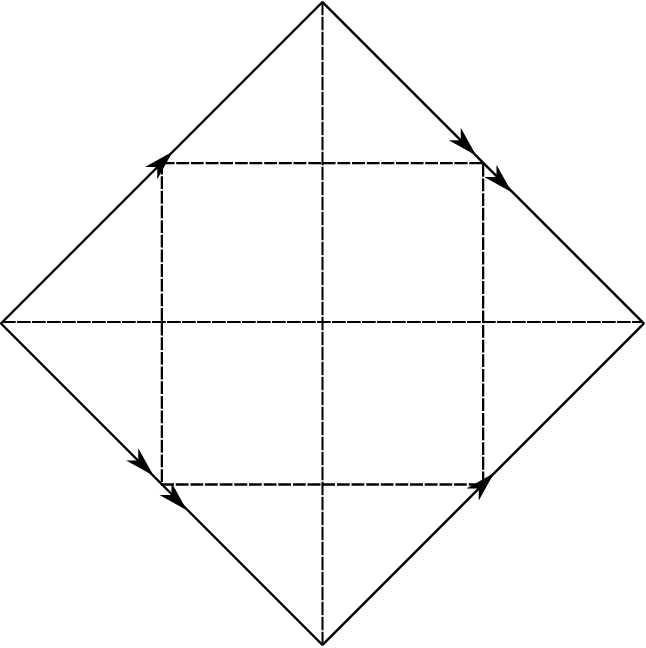}
\caption{The cusp shape of the hyper-surface $\mathcal{H}_F$, together with its tessellation by unit squares (dashed).}\label{fig:cuspshape}
\end{figure}

Notice that the cusp sections of the ambient four-manifold $\mathcal{M}$ can also be generated by translations along the vectors $(2,2,0)$, $(2,-2,0)$ and $(2,0,2)$. In other words, the cusp is obtained by identifying opposite faces of a slanted parallelepiped over a square base. Parallel copies of the base square correspond to cusp sections of the hyper-surface $\mathcal{H}_F$. Each cusp section of $\mathcal{M}$ is intersected by six cusp sections of the hyper-surfaces $\mathcal{H}_F$, $F \in \mathcal{F}(P^4)$, arranged in three parallel pairs, as shown in Figure~\ref{fig:cusptiling}.

\begin{figure}[htbp]
\centering
\includegraphics[width=0.6\textwidth]{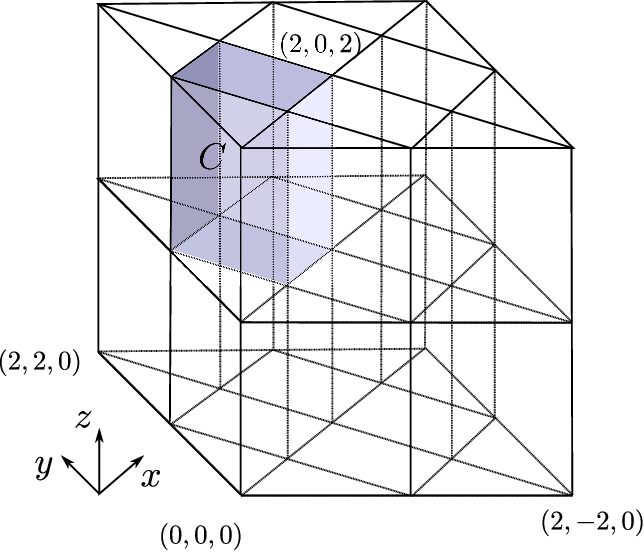}
\caption{A fundamental domain $\calD$ for the cusps of the manifold $\calM$, together with its tessellation into copies of the unit cube $C$ (shaded). The associated lattice is generated by translations along the vectors $(2,2,0)$, $(2,-2,0)$ and $(2,0,2)$. A cusp of a hyper-surface $\calH_F$ corresponds to the intersection of $\calD$ with an affine hyperplane of the form $z=c$, with $c$ integer. }\label{fig:cusptiling}
\end{figure}

It is useful to write down the equations for the linear subspaces which are naturally associated with the ideal vertices of $P^4$. Given such a vertex $v$, the labels on the facets which are adjacent to $v$ generate a co-dimension one linear subspace of $V$. Recall that each of these vertices is naturally labelled with an integer $i\in\{1,2,3,4,5\}$. The linear subspace associated to the vertex labelled $i$ is defined by the equation 
\begin{equation}\label{eq:cuspequation} 
\sum_{j\neq i} x_j=0.
\end{equation}

\section{Mutations}\label{section:mutations}
In this section, we shall show how to use mutations along hyper-surfaces in the manifold $\mathcal{M}$ in order to obtain new manifolds which are not isometric to the original one and have interesting cusp sections. Namely, we shall construct two singled cusped hyperbolic manifolds $\mathcal{X}$ and $\mathcal{Y}$ whose cusp sections are, respectively, an orientable $\mathbb{S}^1$-bundle over the Klein bottle and a three-torus. 

A \textit{mutation} consists in cutting $\calM$ open along a two-sided, totally geodesic hyper-surface $\calH$, choosing an isometry $\phi: \calH \longrightarrow \calH$ from one of the resulting totally geodesic boundary components to the other, and glueing back the two boundary components using the isometry $\phi$. Clearly, this operation can be performed more than once as long as we pick a collection of mutually disjoint hyper-surfaces.

In the case of the manifold $\calM$ constructed in Section \ref{section:symmetric-example}, we can pick at most two disjoint totally geodesic hyper-surfaces corresponding to the facets of the polytope $P^4$ since, by Proposition \ref{prop:colourings-general}, each such a facet $F$ lifts to a unique hyper-surface in the orbifold cover $\mathcal{M}\rightarrow P^4$, and there are at most two mutually disjoint facets in $P^4$.

It is useful to introduce labels on the cusps of $\mathcal{M}$. In order to do so, we choose one of the copies, named $P_0$, of the polytope $P^4$ that tessellates the manifold $\mathcal{M}$. We label its ideal vertices, which naturally correspond to distinct sub-graphs $K_4 \subset K_5$, with the numbers $\{1$, $2$, $3$, $4$, $5\}$. 

The ideal vertices of $P_0$ determine the five cusps of $\mathcal{M}$, which we label $(1,+),(2,+),\dots,(5,+)$. Each other cusp of $\mathcal{M}$ is the image of exactly one of the cusps  $(1,+),(2,+),\dots,(5,+)$ under the action of the group $V < \mathrm{Isom}_c(\mathcal{M})$, c.f. the short exact sequence (\ref{symmetries}). We label these remaining five cusps $(1,-),(2,-),\dots,(5,-)$, respectively.

\subsection{An orientable manifold with a single non-toric cusp}\label{subsection:non-toric-cusp}
Let us pick two disjoint hyper-surfaces $\mathcal{H}_1$ and $\mathcal{H}_2$ in $\calM$, for instance those associated with the facets $F_1$ and $F_2$ having colours $(1,1,1,0,0)$ and $(0,0,1,1,1)$.
To these two hyper-surfaces, there correspond a total of twelve square cusp sections (six to each). Each of the cusps of $\mathcal{H}_i$ intersects exactly one cusp section of $\mathcal{M}$. In particular, the cusps of $\mathcal{H}_1$ (resp. $\mathcal{H}_2$) lie in the cusps of $\calM$ with labels $(1\pm,2\pm,3\pm)$ (resp. $(3\pm,4\pm,5\pm)$). Each cusp section of the ambient manifold $\mathcal{M}$ intersects exactly one of the cusp sections of $\mathcal{H}_1 \cup \mathcal{H}_2$, with the exception of the cusps labelled $(3,+)$ and $(3,-)$, each of which intersects \textit{two} cusp sections of $\mathcal{H}_1 \cup \mathcal{H}_2$.

We shall use mutations along the hyper-surfaces $\mathcal{H}_i$ in order to build an orientable manifold $\calX$ with one non-toric cusp and volume $2\cdot v_m=8\pi^2/3$.
Let us cut the manifold $\mathcal{M}$ along $\mathcal{H}_1$ and $\mathcal{H}_2$ to produce a manifold $\overline{\mathcal{M}}$ with four totally geodesic boundary components $\mathcal{H}_i^{\pm}$, $i=1,2$.  The result of this operation is cutting open the cusp sections of $\mathcal{M}$ along the cusp sections of the hyper-surfaces $\calH_i$ that produces a total of twelve cusp sections. Each of these sections is a Euclidean manifold with totally geodesic boundary, homeomorphic to $T \times [0,1]$, where $T$ is the flat torus represented in Figure~\ref{fig:cuspshape}. 

Recall that we have chosen \emph{a priori} a root for the developing graph of the colouring, \emph{i.e.} a copy $P_0$ of the polytope $P^4$ in the tessellation of $\calM$. We label by $\calH_i^+$ the boundary components of $\overline{\calM}$ that intersects the polytope $P_0$ in the facet $F_i$, $i=1,2$, respectively, and by $\calH_i^-$ the other one.
The cusps $(1,\pm),(2,\pm),(4,\pm),(5,\pm)$, are cut open producing eight of the cusps of $\overline{\calM}$. The cusps coming from $(1,\pm)$ and $(2,\pm)$ (resp. $(4,\pm)$ and $(5,\pm)$) are bounded by the hyper-surface $\mathcal{H}_1^+$ (resp. $\mathcal{H}_2^+$) on one side and by $\mathcal{H}_1^-$ (resp. $\mathcal{H}_2^-$) on the other. These cusps are isometric to the product $T\times[0,2]$. 
The cusps of $\mathcal{M}$ labelled $(3,+)$ and $(3,-)$ are cut open to produce a total of four cusp sections, isometric to the product $T\times[0,1]$ labelled $(3a,\pm)$ and $(3b,\pm)$. Here, the letter $a$ corresponds to the cusp carried by $\calH_1$, and the letter $b$ corresponds to the cusp carried by $\calH_2$. We call these the \emph{short} cusps of $\overline{\calM}$. Each of them is bounded by $\calH^{+}_1$ (resp., $\calH^{-}_1$) on one side, and by $\calH^{+}_2$ (resp., $\calH^{-}_2$) on the other. All other cusps are referred to as \textit{long} ones. 

The $T$-fibres of the short cusps are naturally tessellated by copies of the unit square. Notice that the colouring assigned to the sides of this square is obtained by considering the restriction of the colouring on the cubic vertex figure of $P^4$ labelled $3$ to the four faces with colours different from $(1,1,1,0,0)$ and $(0,0,1,1,1)$, as in Figure~\ref{fig:squarecolor2}.   

\begin{figure}[ht]
\centering
\includegraphics[width=0.4\textwidth]{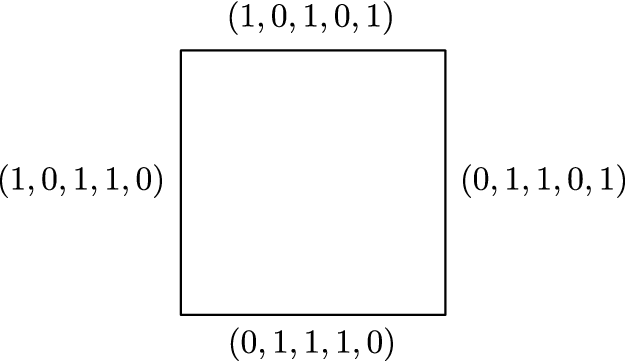}
\caption{The restricition of the $V$-colouring on the vertex figure of the ideal vertex labelled $3$ to a square section parallel to the hyper-surfaces $\calH_1$ and $\calH_2$}\label{fig:squarecolor2}
\end{figure}

These colours generate a co-dimension $2$ subspace $W$ of $V$, defined by the equations: 

\begin{equation}\label{eq:sliceequation}x_1+x_2+x_4+x_5=0,\;x_1+x_2+x_3=0. \end{equation}

Then the four affine subspaces of $\calA$ (see Section \ref{section:colourings-general}, Remark \ref{remark:affine-subspaces}) associated with the subspace $W$ correspond to the cusps labelled $(3a,\pm)$ and $(3b,\pm)$ of $\overline{\calM}$. They are given by the following equations:

\begin{itemize}
\item The cusp $(3a,+)$ corresponds to 
\begin{equation}\label{eq:sliceequation1}x_1+x_2+x_4+x_5=0,\;x_1+x_2+x_3=0. \end{equation}

\item The cusp $(3b,+)$ corresponds to 
\begin{equation}\label{eq:sliceequation2}x_1+x_2+x_4+x_5=0,\;x_1+x_2+x_3=1. \end{equation}

\item The cusp $(3a,-)$ corresponds to 
\begin{equation}\label{eq:sliceequation3}x_1+x_2+x_4+x_5=1,\;x_1+x_2+x_3=0. \end{equation}

\item The cusp $(3b,-)$ corresponds to 
\begin{equation}\label{eq:sliceequation4}x_1+x_2+x_4+x_5=1,\;x_1+x_2+x_3=1. \end{equation}
\end{itemize}
Notice that this labelling is coherent with the one expressed in \eqref{eq:cuspequation}: the cusps of $\calM$ with labels $(3.+)$ and $(3,-)$ correspond to different affine subspaces over $x_1+x_2+x_4+x_5=0$. 

An isometry of $\calM$ induced by summation with a vector $v$ from $V$ is a covering transformation, and restricts to an isometry acting on the lifts of each $F_i$. Since $\calH_i$ is the unique lift of $F_i$, then the isometry induced by $v$ restricts to an isometry of $\calH_i$. Thus, the orientation-reversing isometry $\alpha$ of $\calM$ defined by summation with the vector $(0,0,1,0,0)$ naturally induces an isometry of $\overline{\calM}$ which maps the cusp $(3a,\pm)$ to $(3b,\pm)$. The isometry $\beta:\calM\rightarrow\calM$ defined by summation with the vector $(1,1,0,1,0)$ induces an isometry of $\overline{\calM}$, mapping the cusp $(3a,+)$ to $(3a,-)$ and $(3b,+)$ to $(3b,-)$.

In order to understand which boundary components of $\overline{\calM}$ bound each of the short cusps, we notice that the isometry $\alpha$ is an orientation reversing isometry of $\calM$ which is orientation-preserving on the hyper-surfaces $\calH_i$. Indeed, $\alpha$ is induced by the vector
\begin{equation*}
(0,0,1,0,0) = \underbrace{(1,1,0,0,1)}_{v_1} + \underbrace{(1,1,0,1,0)}_{v_2} + \underbrace{(0,1,1,0,1)}_{v_3} + \underbrace{(1,0,1,1,0)}_{v_4} + \underbrace{(1,1,1,0,0)}_{w}, 
\end{equation*}
where $v_i$'s belong to the colours of the neighbours of $F_1$ (as shown in Figure~\ref{fig:coloreddarts} on the left), and $w$ is the colour of $F_1$ itself. Thus, in the induced colouring of $F_1$ we have a sum of four colours from $V\diagup \langle w \rangle$, and the isometry $\alpha$ of $\calH_1$ can be expressed as a composition of four reflections, which is orientation preserving on $\calH_1$. 

We can apply an analogous argument to show that $\alpha$ is orientation-preserving on $\calH_2$, as well. Here, we use the fact that 
\begin{equation*}
(0,0,1,0,0) = \underbrace{(1,0,0,1,1)}_{v_1} + \underbrace{(0,1,0,1,1)}_{v_2} + \underbrace{(1,0,1,1,0)}_{v_3} + \underbrace{(0,1,1,0,1)}_{v_4} + \underbrace{(0,0,1,1,1)}_{w}, 
\end{equation*}
where $v_i$'s belong to the colours of the neighbours of $F_2$ (as shown in Figure~\ref{fig:coloreddarts} on the right), and $w$ is the colour of $F_2$ itself.

Therefore, the induced isometry of $\overline{\calM}$ permutes the boundary components $\calH_i^+$ and $\calH_i^-$. 

The map $\beta$ instead, is an orientation-reversing isometry of $\calM$ which is orientation-reversing on $\calH_1$ and orientation-preserving on $\calH_2$. Indeed
\begin{equation*}
(1,1,0,1,0) = \underbrace{(1,0,0,1,1)}_{v_1} + \underbrace{(0,1,1,1,0)}_{v_2} + \underbrace{(0,0,1,1,1)}_{w}
\end{equation*}
where all the vectors on the right-hand side belong to the colours of the neighbours of $F_1$ (see Figure~\ref{fig:coloreddarts}, left), on one hand, and the $v_i$'s belong to the colours of the neighbours of $F_2$ (see Figure~\ref{fig:coloreddarts}, right), on the other, while $w$ is the colour of $F_2$ itself. 

The colour $w$ is the colour of $F_2$ itself. Therefore the induced map on $\overline{\calM}$ \emph{fixes} the boundary components $\calH_1^{\pm}$ and permutes $\calH_2^+$ and $\calH_2^-$. As a consequence, the short cusps of $\overline{\calM}$ have the following boundary components 
\begin{itemize}
\item The cusp $(3a,+)$ is bounded by $\calH_1^+$ on one side and by $\calH_2^+$ on the other.
\item The cusp $(3b,+)$ is bounded by $\calH_1^-$ on one side and by $\calH_2^-$ on the other.
\item The cusp $(3a,-)$ is bounded by $\calH_1^+$ on one side and by $\calH_2^-$ on the other.
\item The cusp $(3b,-)$ is bounded by $\calH_1^-$ on one side and by $\calH_2^+$ on the other.
\end{itemize}

Let us define a self-isometry of $\mathcal{H}_1$ that we shall use in our construction. Consider the short exact sequence (\ref{hypersurfaces}). Pick the isometry $\phi_1=v\circ\sigma_1$, where $\sigma_1\in \text{Symm}\,(P^3)$ exchanges the hyperbolic vertices labelled $4$ and $5$ and induces the cycle $(1,2,3)$ on the ideal vertices of Figure~\ref{fig:facets}, and $v=[(0,1,1,0,1)]\in V'$. The translation by $v$ corresponds to a reflection in the facet corresponding to the complete sub-graph $K_3 \subset K_5$ spanned by the vertices labelled $2$, $3$, and $4$ in Figure~\ref{fig:facets}.

Similarly, let us define a self-isometry of $\mathcal{H}_2$ by $\phi_2=\sigma_2$, where $\sigma_2 \in \text{Symm}(P^3)$ fixes the hyperbolic vertices labelled $1$ and $2$, and induces the cycle $(3,4,5)$ on the ideal ones. Notice that both $\phi_1$ and $\phi_2$ are orientation-preserving isometries of $\mathcal{H}_1$ and $\mathcal{H}_2$, respectively. Moreover, the maps $\phi_1$ and $\phi_2$ induce the following cycles on the cusps of $\mathcal{M}$:

\begin{equation}\label{eq:cycle1}\phi_1: (1,+)\rightarrow (2,+) \rightarrow (3,+)\rightarrow (1,-)\rightarrow (2,-)\rightarrow (3,-) \rightarrow (1,+);\end{equation}
\begin{equation}\label{eq:cycle2}\phi_2:(3,\pm)\rightarrow (4,\pm)\rightarrow (5,\pm)\rightarrow (3\pm).\end{equation}

We glue isometrically in pairs the boundary components of $\overline{\mathcal{M}}$ using the induced maps $\phi_i:\mathcal{H}_i^+\rightarrow \mathcal{H}_i^-$ to produce a manifold $\calX$. 
 
\begin{prop}\label{prop:onecusp}
The manifold $\calX$ is an orientable, finite-volume, single-cusped hyperbolic manifold. The cusp shape is a torus-bundle over $\mathbb{S}^1$, with the monodromy map given by a hyper-elliptic involution.
\end{prop}

\begin{proof}
We prove that the cusp sections of $\overline{\mathcal{M}}$ are glued together along their boundaries as in Figure~\ref{fig:localcusp} to produce a unique cycle of maximal length:

\begin{align*}\label{eq:cuspcycle}
(3a,+)\xrightarrow{\phi_1}(1,-)\xrightarrow{\phi_1}(2,-)\xrightarrow{\phi_1} (3b,-)\xrightarrow{\phi_2}(4,-)\xrightarrow{\phi_2}(5,-)\xrightarrow{\phi_2}(3a,-)\xrightarrow{\phi_1}(1,+)\xrightarrow{\phi_1}\dots\\\dots\xrightarrow{\phi_1} (2,+)\xrightarrow{\phi_1}(3b,+)\xrightarrow{\phi_2^{-1}}(5,+)\xrightarrow{\phi_2^{-1}}(4,+)\xrightarrow{\phi_2^{-1}}(3a,+).
\end{align*} 
corresponding to the \emph{unique} cusp $\calC$ of the manifold $\calX$.

\begin{figure}[ht]
\centering
\includegraphics[width=0.75\textwidth]{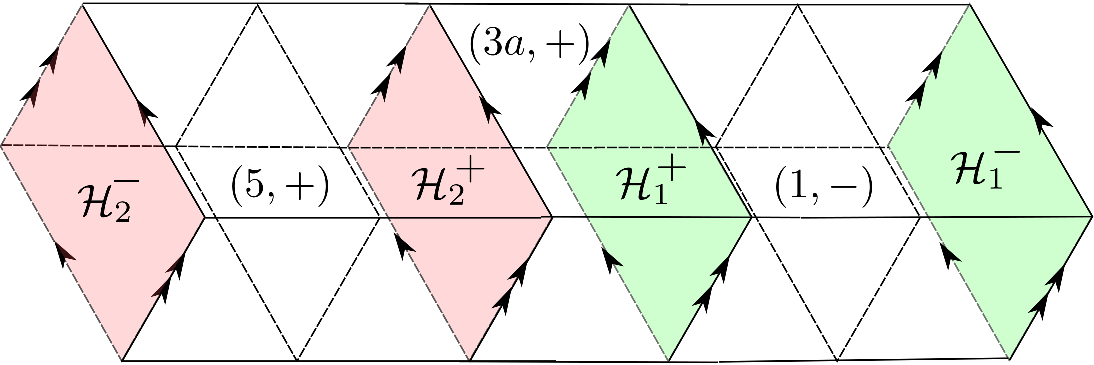}
\caption{The local cusp structure of the manifold \calX, as a result of cusp sections glueing in $\overline{\calM}$. Coloured regions represent intersections with the hyper-surfaces $\calH_1$ and $\calH_2$, and are each isometric to the torus $T$ in Figure~\ref{fig:cuspshape}. Notice that the cusps with labels $(3a,\pm)$ and $(3b,\pm)$ are shorter than others, and are bounded by different hyper-surfaces on the sides. This property is crucial, because it allows us to  ``mix'' the cycles induced on the cusps by the maps $\phi_1$ and $\phi_2$.}\label{fig:localcusp}
\end{figure}

It is clear \emph{a priori} that, no matter how many cusps the manifold $\calX$ has, the corresponding cusp sections decompose into a union of $20$ copies of the product $T\times[0,1]$, glued along their boundaries in cycles to form a certain amount of closed Euclidean manifolds. Each such a cycle corresponds to a single cusp of $\calX$, and the shape of each cusp section is a torus bundle over $\mathbb{S}^1$. In order to compute the cycles, we choose a copy of $T\times[0,1]$, and perform a sequence of reflections in parallel copies of the torus $T$ of the form $T\times\{0\}$ and $T\times\{1\}$ at unit distance one from the other. All these reflections are induced by coloured isometries of the manifold $\calM$. In fact, there are two kinds of toric fibres in the cusp $\calC$: those which separate different cusps of the manifold $\overline{\calM}$, and the central slices of the form $T\times\{1\}$ of the long cusps with labels $(1\pm),(2,\pm),(4\pm),(5,\pm)$. Each of the latter cusps is isometric to $T\times[0,2]$.

Each of the maps $\phi_i:\calH_i\rightarrow \calH_i$ admits a unique orientation-reversing extension $\widetilde{\phi_i}$ to an isometry in $\text{Isom}_c(M)\cong V\rtimes \mathfrak{S}_5$ as follows:

\begin{equation}\label{eq:extension}
\widetilde{\phi_1}:\left(\begin{array}{c}
a\\
b\\
c\\
d\\
e\\
\end{array}\right)\rightarrow
\left(\begin{array}{c}
c+1\\
a\\
b\\
e\\
d+1\\
\end{array}\right),\;
\widetilde{\phi_2}:\left(\begin{array}{c}
a\\
b\\
c\\
d\\
e\\
\end{array}\right)\rightarrow
\left(\begin{array}{c}
a\\
b\\
e+1\\
c+1\\
d+1\\
\end{array}\right)
\end{equation}
Here we use the affine coordinates on the vertex set $\calA$ of the developing graph $\Gamma(\calM)$. These isometries of $\calM$ induce reflections along the copies of $T$ which separate different cusps of $\overline{\calM}$. 

The corresponding reflections along the central sections of the long cusps of $\overline{\calM}$ labelled $i\in \{1,2,4,5\}$, can be expressed by the maps $R_i$ defined, in the affine coordinates on $\mathcal{A}$, as follows:

\begin{align}\label{eq:reflections}
R_1:\left(\begin{array}{c}
a\\
b\\
c\\
d\\
e\\
\end{array}\right)\rightarrow
\left(\begin{array}{c}
a+1\\
b\\
c\\
d+1\\
e+1\\
\end{array}\right),&\;\;
R_2:\left(\begin{array}{c}
a\\
b\\
c\\
d\\
e\\
\end{array}\right)\rightarrow
\left(\begin{array}{c}
a\\
b+1\\
c\\
d+1\\
e+1\\
\end{array}\right)\\
R_4:\left(\begin{array}{c}
a\\
b\\
c\\
d\\
e\\
\end{array}\right)\rightarrow
\left(\begin{array}{c}
a+1\\
b+1\\
c\\
d+1\\
e\\
\end{array}\right),&\;\;
R_5:\left(\begin{array}{c}
a\\
b\\
c\\
d\\
e\\
\end{array}\right)\rightarrow
\left(\begin{array}{c}
a+1\\
b+1\\
c\\
d\\
e+1\\
\end{array}\right)
\end{align}
Notice that the map $R_i$ is simply a reflection in the hyper-surface which intersect the cusps of $\calM$ labelled $i \in \{1, 2, 4, 5\}$ and is parallel to $\calH_1\cup \calH_2$. To each copy of $T\times [0, 2]$, there corresponds a co-dimension two affine subspace of $\calA$, determined by the eight copies of the polytope $P^4$ that tessellate the former. 

Since each reflection acts on the affine space $\calA$, the cycles induced on the $20$ copies of $T\times [0,1]$ naturally correspond to the cycles induced on the corresponding affine subspaces. Moreover, to each cycle $\calC$ there corresponds naturally an element $\phi_{\calC}$ in the coloured isometry group of the torus $T$, which can be described by the following short exact sequence:
\begin{equation}\label{eq:isometries-torus} 
0\rightarrow W\rightarrow \text{Isom}_c(T)\rightarrow D_4\rightarrow 0.
\end{equation}
The isometry defined by $\phi_{\calC}$  corresponds to the monodromy of the fibration $$0\rightarrow T\rightarrow \calC\rightarrow \mathbb{S}^1\rightarrow 0.$$ 
 
We define the isometries $\psi_1, \psi_2:\calM \rightarrow \calM$ as follows:

\begin{align}\label{eq:psi1}
\psi_1=\widetilde{\phi_1}\circ R_2\circ\widetilde{\phi_1}\circ R_1\circ\widetilde{\phi_1}\\
\psi_2=\widetilde{\phi_2}\circ R_5\circ\widetilde{\phi_2}\circ R_4\circ\widetilde{\phi_2}
\end{align}

By substituting the respective affine expressions for each factor in the composition of maps given by (\ref{eq:extension}) -- (\ref{eq:reflections}), we represent $\psi_1$ and $\psi_2$ as
\begin{equation}\label{eq:psiaffine}
\psi_1:\left(\begin{array}{c}
a\\
b\\
c\\
d\\
e\\
\end{array}\right)\rightarrow
\left(\begin{array}{c}
a+1\\
b+1\\
c+1\\
e+1\\
d\\
\end{array}\right),\;
\psi_2:\left(\begin{array}{c}
a\\
b\\
c\\
d\\
e\\
\end{array}\right)\rightarrow
\left(\begin{array}{c}
a\\
b\\
c+1\\
d+1\\
e+1\\
\end{array}\right)
\end{equation}

Let us choose a cusp of $\overline{\calM}$, \textit{e.g.} labelled $(3a,+)$. The map $\psi_1$ corresponds to a composition of five subsequent reflections along parallel copies of the torus $T$ in the cusp $C$ of $\calX$, where the first reflection is along the boundary component of $(3a,+)$ corresponding to the hyper-surface $\calH_1^+$. Note that $\psi_1$ maps the cusp $(3a,+)$ to $(3b,-)$, and it does so by sending $\calH_1^+$ to $\calH_1^-$ and \emph{fixing} $\calH_2^+$. 

Now we apply the map $\psi_2$, \emph{i.e.} we perform other five subsequent reflections, starting from the boundary component of $(3b,-)$ corresponding to $\calH_2^+$. This maps $(3b,-)$ to $(3a,-)$, $\calH_1^-$ to $\calH_1^+$ and $\calH_2^-$ to $\calH_2^+$. Again, we apply $\psi_1$, mapping $(3a,-)$ to $(3b,+)$, $\calH_1^+$ to $\calH_1^-$ and fixing $\calH_2^-$. 

This time, we arrive to a cusp which is bounded by $\calH_2^-$ on one side. Therefore we need to consider the pairing map $\phi_2^{-1}: \calH_2^-\rightarrow \calH_2^+$ to start with. Thus, the last five reflection correspond to applying $\psi_2^{-1}$ in order to return back to $(3a,+)$ and close up the cycle.

The monodromy map $\phi:T\rightarrow T$ can therefore be expressed as the following composition of maps in $\text{Aut}_c(M)$:

\begin{equation}\label{eq:composition}
\phi = \psi_2^{-1}\circ\psi_1\circ\psi_2\circ\psi_1 
\end{equation}

By substituting the corresponding expressions in the affine coordinates from (\ref{eq:psiaffine}), we get

\begin{equation}\label{eq:monodromy1}
\phi:\left(\begin{array}{c}
a\\
b\\
c\\
d\\
e\\
\end{array}\right)\rightarrow
\left(\begin{array}{c}
a\\
b\\
c\\
d+1\\
e+1\\
\end{array}\right).
\end{equation}

As expected, this affine map fixes the affine space $W$ associated with the slice $T$, defined by formula \eqref{eq:sliceequation}. In fact, in the affine coordinates, $\phi$ is a translation by vector $(0,0,0,1,1)$. Notice that this vector is expressed as a sum of the colours assigned to the sides of the square in Figure~\ref{fig:squarecolor2} adjacent to its upper left vertex. Therefore, the resulting monodromy map is realised by a composition of two reflections along two orthogonal axes: it is a hyper-elliptic involution of the torus $T$.  The lattice associated with the cusp $\calC$ is generated by the translations along the vectors $(2,2,0)$ and $(2,-2,0)$ together with the Euclidean isometry $$\psi: (x,y,z)\rightarrow(-x,-y,z+20).$$ 
\end{proof} 

\subsection{An orientable manifold of small volume with a single toric cusp}\label{subsection:toric-cusp}

Let us step back for a moment, and look once more at the symmetric manifold $\mathcal{M}$ defined in Section \ref{subsection:colouring-definition}. There is an orientation-preserving involution $i$, exchanging the hyper-surfaces $\mathcal{H}_1$ and $\mathcal{H}_2$ associated with the facets having colours $(1,1,1,0,0)$ and $(0,0,1,1,1)$. In terms of the short exact sequence (\ref{symmetries}), it is given by $i=\sigma$, where $\sigma=(1,4)(2,5)\in \mathfrak{S}_5$.

Again, we use mutations along the hyper-surfaces $\calH_1$ and $\calH_2$ in order to pair all the cusps of $\overline{\calM}$. For the hyper-surface $\calH_1$, we choose the isometry $\phi_1:\mathcal{H}_1^+\rightarrow \mathcal{H}_1^-$ introduced in Section \ref{subsection:non-toric-cusp}. For the hyper-surface $\calH_2$, we notice that the involution $i$ sends $\calH_1^{+}$ (resp., $\calH_1^{-}$) to $\calH_2^{+}$ (resp., $\calH_2^{-}$), and we conjugate the map $\phi_1$ by the involution $i$. Finally, the mutations $\phi_1$ and $\phi_2$ produce a hyperbolic manifold which we denote by $\calY$.

\begin{prop}\label{prop:single-toric-cusp}
The manifold $\calY$ is an orientable finite-volume hyperbolic $4$-manifold with a single cusp, whose cross-section is a three-dimensional torus. 
\end{prop}

\begin{proof}
In order to determine the homeomorphism type of the cusp of $\calY$, we follow the same procedure as in the proof of Proposition \ref{prop:onecusp}. The affine expression for the orientation-reversing lift $\widetilde{\phi_1}$ of $\phi_1$ is the same as before, and is given in \eqref{eq:extension}. The affine expression for the orientation-reversing lift of the second mutation $\phi_2$ is

\begin{equation}\label{eq:newaffine}
\widetilde{\phi}_2:\left(\begin{array}{c}
a\\
b\\
c\\
d\\
e\\
\end{array}\right)\rightarrow
\left(\begin{array}{c}
b\\
a+1\\
e\\
c+1\\
d\\
\end{array}\right).
\end{equation}

One can check that these pairing maps arrange the cusps of $\overline{\calM}$ in a single cycle:
\begin{align}\label{eq:cuspcycle2}
(3a,+)\xrightarrow{\phi_1}(1,-)\xrightarrow{\phi_1}(2,-)\xrightarrow{\phi_1} (3a,-)\xrightarrow{\phi_2}(4,+)\xrightarrow{\phi_2}(5,+)\xrightarrow{\phi_2}(3b,+)\xrightarrow{\phi_1^{-1}}(2,+)\xrightarrow{\phi_1^{-1}}\dots\\\dots\xrightarrow{\phi_1^{-1}} (1,+)\xrightarrow{\phi_1^{-1}}(3b,-)\xrightarrow{\phi_2^{-1}}(5,-)\xrightarrow{\phi_2^{-1}}(4,-)\xrightarrow{\phi_2^{-1}}(3a,+).
\end{align}

Let us define the maps $\psi_1$ and $\psi_2$ as the following compositions:
\begin{align}\label{eq:psi2new}
\psi_1=\widetilde{\phi}_1\circ R_2 \circ \widetilde{\phi}_1 \circ R_1\circ \widetilde{\phi}_1,\\
\psi_2=\widetilde{\phi}_2\circ R_5 \circ \widetilde{\phi}_2 \circ R_4\circ \widetilde{\phi}_2,
\end{align} 
The affine expression for $\psi_1$ coincides with the one given by formula \eqref{eq:psiaffine} in Section \ref{subsection:non-toric-cusp}. The affine expression for $\psi_2$ is given below
\begin{equation}\label{eq:affinepsi2new}
\psi_2:\left(\begin{array}{c}
a\\
b\\
c\\
d\\
e\\
\end{array}\right)\rightarrow
\left(\begin{array}{c}
b+1\\
a\\
c+1\\
d+1\\
e+1\\
\end{array}\right).
\end{equation}

Then the monodromy map $\phi$ of the cusp section of $\calY$ equals 
\begin{equation}\label{eq:newcomposition}
\phi = \psi_2^{-1}\circ\psi_1^{-1}\circ\psi_2\circ\psi_1.
\end{equation}

By substituting the affine expressions for $\psi_1$ and $\psi_2$ given in \eqref{eq:psiaffine} and \eqref{eq:affinepsi2new}, we see that $\phi$ is actually the identity map of the torus $T$.
As a consequence, the manifold $\calY$ has a single toric cusp $\calC$. The Euclidean cusp shape is determined by translations along the vectors $(2,2,0)$, $(2,-2,0)$ and $(0,0,20)$.
\end{proof}

\begin{rem}
The examples above answer affirmatively, in the orientable case, to Question 4.15 from \cite{KS2014}.
\end{rem}

\bigskip

\begin{flushleft}
\textit{Alexander Kolpakov\\
Department of Mathematics\\
40 St. George Street\\
Toronto ON\\
M5S 2E4 Canada\\}
\texttt{kolpakov dot alexander at gmail dot com}
\end{flushleft}

\medskip

\begin{flushleft}
\textit{Leone Slavich\\
Dipartimento di Matematica \\
Largo Bruno Pontecorvo 5\\
56127 Pisa, Italy\\}
\texttt{leone dot slavich at gmail dot com}
\end{flushleft}
 
\end{document}